\documentclass[12pt]{amsart}

\usepackage{amscd,amssymb,amsopn,amsmath,amsthm,mathrsfs,graphics,amsfonts,enumerate,verbatim,calc,pb-diagram
}

\usepackage[all]{xy}

\usepackage{color}

\usepackage[OT2,OT1]{fontenc}
\newcommand\cyr{%
\renewcommand\rmdefault{wncyr}%
\renewcommand\sfdefault{wncyss}%
\renewcommand\encodingdefault{OT2}%
\normalfont
\selectfont}
\DeclareTextFontCommand{\textcyr}{\cyr}

\usepackage{amssymb,amsmath}

\DeclareFontFamily{OT1}{rsfs}{}
\DeclareFontShape{OT1}{rsfs}{n}{it}{<-> rsfs10}{}
\DeclareMathAlphabet{\mathscr}{OT1}{rsfs}{n}{it}

\numberwithin{equation}{section}
\newtheorem{theorem}{Theorem}[section]
\newtheorem{lemma}[theorem]{Lemma}
\newtheorem{proposition}[theorem]{Proposition}
\newtheorem{corollary}[theorem]{Corollary}

\newtheorem{problem}{Problem}

\theoremstyle{definition}
\newtheorem{definition}[theorem]{Definition}
\newtheorem{remark}[theorem]{Remark}
\theoremstyle{remark}
\newtheorem{observation}[theorem]{Observation}
\newtheorem{example}[theorem]{Example}

\newtheorem{acknowledgement}{Acknowledgement}

\newcommand{\Ass}{\operatorname{Ass}}

\newcommand{\Ann}{\operatorname{Ann}}

\newcommand{\fm}{\frak{m}}
\newcommand{\fp}{\frak{p}}
\newcommand{\fq}{\frak{q}}

\begin{document}
\title[Limit closure]{On the limit closure of a sequence of elements\\ in local rings}

\author[Nguyen Tu Cuong]{Nguyen Tu Cuong}
\address{Institute of Mathematics, Vietnam Academy of Science and Technology, 18 Hoang Quoc Viet, Hanoi,
Vietnam}
\email{ntcuong@math.ac.vn}

\author[Pham Hung Quy]{Pham Hung Quy}
\address{Department of Mathematics, FPT University, Hoa Lac Hi-Tech Park, Hanoi, Vietnam}
\email{quyph@fe.edu.vn}

\thanks{
This work is partially supported by a fund of Vietnam National Foundation for Science
and Technology Development (NAFOSTED) under grant number 101.04-2020.10.}


 
\begin{abstract} We present a study of the limit closure $(\underline{x})^{\lim}$ of a system of parameters $\underline{x}$ in local rings. Firstly, we answer the question about elements which are always contained in the limit closure of a system of parameters. Then we apply it to give a characterization of systems of parameters which is a generalization of previous results of Dutta and Roberts in \cite{DR} and of Fouli and Huneke in \cite{FH}. We also prove a topological characterization of unmixed local rings. In two dimensional case, we compute explicitly the limit closure of a system of parameters. Some interesting examples are given.
\end{abstract}

\maketitle

\section{Introduction}
System of parameters is a basic concept of local algebra. Let $(R, \fm)$ be a Noetherian local ring of dimension $t$ and $\underline{x} = x_1, \ldots, x_t$ a system of parameters. Understanding the relations of elements in a system of parameters is one of the most important problems in commutative algebra. Indeed, Hochster asked about a "simple" relation that cannot be satisfied by a system of parameters (cf. \cite{Ho}). This question is called the monomial conjecture and stated as follows. For for all $n \ge 1$ we have $(x_1 \ldots x_t)^n \notin (x_1^{n+1}, \ldots, x_t^{n+1})$. Hochster proved the monomial conjecture when $R$ contains a field. Recently, Andr\'{e} \cite{An18} confirmed this conjecture in the full generality. Moreover it is easy to see that
$$(x_1^2, \ldots, x_t^2) : (x_1 \ldots x_t) \subseteq (x_1^3, \ldots, x_t^3) : (x_1 \ldots x_t)^2 \subseteq \cdots \subseteq (x_1^{n+1}, \ldots, x_t^{n+1}) : (x_1 \ldots x_t)^n \subseteq \cdots.$$
Thus it is natural to consider the following
$$(x_1, \ldots, x_t)^{\lim} := \bigcup_{n>0}\big((x_1^{n+1}, \ldots, x_t^{n+1}): (x_1 \ldots x_t)^n \big).$$
 We call $(x_1, \ldots, x_t)^{\lim}$ (or $(\underline{x})^{\lim}$) the {\it limit closure} of the sequence $\underline{x} = x_1, \ldots, x_t$. By the monomial theorem of Hochster and Andr\'{e} we know that $1$ {\it cannot} be contained in $(x_1, \ldots, x_t)^{\lim}$ for any system of parameters $\underline{x} = x_1, \ldots, x_t$.

It is worth to note that if $R$ is Cohen-Macaulay then $(x_1, \ldots, x_t)^{\lim} = (x_1, \ldots, x_t)$ and the converse holds true (cf. \cite{Ha,CHL}). The motivation of our paper is a problem which can be seen as the opposite perspective of the monomial conjecture: Determine elements which are {\it always} contained in $(x_1, \ldots, x_t)^{\lim}$ for all systems of parameters $\underline{x} = x_1, \ldots, x_t$.
For convenience we shall consider this problem for modules. Let $(R,\frak m)$ be a Noetherian local ring, $M$ a finitely
generated $R$-modules of dimension $d$. Let $\underline{x} = x_1, \ldots, x_r$ be a sequence of $r$ elements
of $R$. Then the {\it limit closure} of the sequence $\underline{x}$ in $M$ is a submodule of $M$ defined by
$$(\underline{x})_M^{\lim} = \bigcup_{n>0}\big((x_1^{n+1}, \ldots, x_r^{n+1})M: (x_1\ldots x_r)^n \big).$$
The following problem is the starting point of this work.

\begin{problem}\label{problem} Let $(R,\frak m)$ be a Noetherian local ring, $M$ a finitely
generated $R$-module of dimension $d$. What is
$$\bigcap_{\underline{x}}(x_1, \ldots, x_d)_M^{\lim},$$
where $\underline{x} = x_1, \ldots, x_d$ runs over all systems of parameters of $M$?
\end{problem}
We will show that the above intersection can be interpreted by the primary decomposition of the zero submodule of $M$. Let $(0) = \cap_{\frak p \in \mathrm{Ass}\,M}N(\frak p)$ be a reduced primary decomposition of the zero submodule of $M$. The {\it unmixed component} $U_M(0)$ of $M$  is a submodule defined by
$$U_M(0) =
\bigcap_{\frak p \in \mathrm{Ass}M, \dim R/\fp = d}N(\frak p)$$
It should be noted that $U_{M}(0)$ is just the largest submodule of $M$ of dimension less than $\dim M = d$. The answer for Problem \ref{problem} is as follows.

\begin{theorem}\label{thmA}
Let $(R, \fm)$ be a Noetherian local ring and $M$ a finitely generated $R$-module of dimension $d$. Let $\underline{x} = x_1,
\ldots,x_d$ be a system of parameters of $M$. Then
$$\bigcap_{n>0}(x_1^n, \ldots, x_d^n)_M^{\lim} = U_M(0).$$
\end{theorem}
This intersection formula has many applications. Firstly, we generalize the previous works of Dutta and Roberts \cite{DR} and of Fouli and Huneke \cite{FH} about relation which is satisfied by systems of parameters. Let $\underline{x} = x_1, \ldots, x_t$ be a system of parameters of $R$ and $\underline{y} = y_1, \ldots, y_t$ a sequence of elements such that $(\underline{y}) \subseteq (\underline{x})$. We have a matrix $A = (a_{ij})$, $a_{ij} \in R$, $1 \leq i,j
\leq r$ such that $y_i= \sum_{j=1}^n a_{ij}x_j$, it means
$\mathbf{y} = A \mathbf{x}$, where $\mathbf{x}$ (res. $\mathbf{y}$)
denotes the column vector with entries $x_1, \ldots, x_r$ (res. $y_1, \ldots, y_r$). We abbreviate it by writing
$(\underline{y}) \overset{A}{\subseteq} (\underline{x})$. It easily follows from Crammer's rule that
$\det (A) . (\underline{x}) \subseteq (\underline{y})$. Therefore, we
obtain a determinantal map
$$\det (A)  \, :\, R/(\underline{x}) \to R/(\underline{y}), \quad  m + (\underline{x})
\mapsto \det (A)  m + (\underline{y}).$$
When $R$ is Cohen-Macaulay, Dutta and Roberts in \cite{DR} proved that $\underline{y}$ is a system of parameters if and only if the determinantal map $\det(A)$ is injective. In \cite{FH} Fouli and Huneke extended this result for any local ring by using $(\underline{x})^{\lim}$ instead of $(\underline{x})$. By \cite[5.1.15]{S} we also
have a homomorphism
$$\det (A)  \, :\, R/(\underline{x})^{\lim} \to R/(\underline{y})^{\lim},$$
which is independent of the choice of the matrix $A$. The following is a generalization of Fouli and Huneke's result.

\begin{theorem}\label{thmC}
  Let $(R,\frak m)$ be a catenary equidimensional local ring of dimension $t$. There
exists a positive integer $\ell$, which depends only on $R$, with property: whenever $\underline{x} =
x_1, \ldots ,x_t$ is a system of parameters of $R$ with
$(\underline{x}) \subseteq \frak m^\ell$ and $\underline{y}
=y_1, \ldots,y_t$ a sequence of elements such that $(\underline{y})
\overset{A}{\subseteq} (\underline{x})$ the following statements are
equivalent:
\begin{enumerate}[{(i)}]\rm
    \item {\it $\underline{y}$ forms a system of parameters of $R$;}
    \item {\it The determinantal map $R/(\underline{x})^{\lim} \xrightarrow{\det
A} R/(\underline{y})^{\lim}$ is injective.}
\end{enumerate}
\end{theorem}
As another application of Theorem \ref{thmA}, we give a characterization of {\it unmixed} local rings. In local algebra we often pass to the $\fm$-adic completion $\widehat{R}$ to inherit several good properties of complete local rings. A local ring $(R, \fm)$ is called {\it unmixed} (in sense of Nagata) if $\dim \widehat{R}/\frak P = \dim \widehat{R}$ for all $\frak P \in \Ass \widehat{R} $ i.e. $U_{\widehat{R}}(0) = 0$. Almost local domains in commutative algebra are unmixed. However Nagata in \cite[Example 2, pp. 203--205]{N} constructed a non-unmixed local domain. The following is a characterization of unmixed local rings in terms of the topology defined by limit closures.

\begin{theorem}\label{thmD}
Let $(R,\frak m)$ be a Noetherian local ring of dimension $t$. Then $R$ is
unmixed if and only if the $\frak m$-adic topology is equivalent to the topology defined by $\{(x_1^n, \ldots, x_t^n)^{\lim}\}_{n \ge 1}$ for any system of parameters $\underline{x} = x_1, \ldots, x_t$ of $R$.
\end{theorem}

It should be noted that the limit closure is very complicated to compute. In fact by the works of Hochster and Andr\'{e} we know $(\underline{x})^{\lim} \subseteq \fm$ or $\ell(R/(\underline{x})^{\lim})>0$ for any system of parameters $\underline{x}$. In the two last sections of this paper we give some explicit computations for limit closures. By Theorem \ref{thmA} and passing to $\widehat{R}$, if necessary, we always can reduce to the case of unmixed rings. When $\dim R = 2$ we prove the following result.
\begin{theorem}\label{thmE} Let $(R, \frak m)$ be an unmixed local ring of dimension $d=2$ with the $S_2$-ification $S$. Let $x, y$ be a system of parameters $R$. Then we have the following.
\begin{enumerate}[{(i)}]\rm
\item $(x,y)^{\lim} = (x,y)S \cap R$.
\item $\ell (R/(x,y)^{\lim}) = e(x,y;R) - \ell (H^1_{\frak m}(R)/(x,y)H^1_{\frak m}(R))$, where $e(x,y;R)$ is the multiplicity of $(x,y)$.
\item $\ell ((x, y)^{\lim}/(x,y)) = \ell (H^1(x,y;H^1_{\frak m}(R)))$.
\end{enumerate}
\end{theorem}
The paper is organized as follows. In Section 2 we prove some useful properties of limit closure. The main technique is understanding the limit closure via the canonical map from Koszul cohomology to local cohomology. Then the vanishing theorems of local cohomology is very useful to prove Theorem \ref{thmA}. Section 3 is devoted for Theorem \ref{thmA} and its variants. We prove Theorem \ref{thmC} in Section 4. We also provide an example to claim that the catenary condition is essential. Theorem \ref{thmD} is proved in Section 5. In Section 6 we first consider the relation between of limit closures of a system of parameters in $R$ and its $S_2$-ification. Then we apply the obtained result to prove Theorem \ref{thmE}. In the last Section we compute an explicit example of a limit closure. 

Throughout this paper, $R$ is a commutative Noetherian ring and $M$ is a finitely generated $R$-module. The set of associated primes of $M$ is denoted by $\Ass M$. We also denote $\mathrm{Assh}M = \{\fp \in \Ass M \mid \dim R/\fp = \dim M\}$. For a sequence of elements $\underline{x} = x_1, \ldots, x_r$ and a positive integer $n$ we denote by $\underline{x}^{[n]}$ the sequence $x_1^n, \ldots, x_r^n$. About concepts of commutative algebra we follow \cite{BH98,Mat}. For local cohomology we refer to \cite{BS98}. 

\begin{acknowledgement} The author is deeply grateful to the referee for her/his careful reading and many value suggestions. This paper was written while the authors visited Vietnam Institute for Advanced Study in Mathematics (VIASM), they would like to thank VIASM for the very kind support and hospitality.
\end{acknowledgement}
\section{Basic properties}
Throughout this section, $R$ is a Noetherian ring, and $M$ is a
finitely generated $R$-modules. Let $\underline{x}
= x_1, \ldots, x_r$ be a sequence of $r$ elements of $R$. The following is the main object of this paper.
\begin{definition}[\cite{Hu}] \rm The {\it limit closure} of the sequence $\underline{x}$ in $M$ is a submodule of $M$ defined by
$$(\underline{x})_M^{\lim} = \bigcup_{n>0}\big((\underline{x}^{[n+1]})M: (x_1 \ldots x_r)^n \big),$$
when $M = R$ we write $(\underline{x})^{\lim}$ for short.
\end{definition}
The notion of limit closure is well-defined since
$$(\underline{x}^{[2]})M: (x_1 \ldots x_r) \subseteq (\underline{x}^{[3]})M: (x_1 \ldots x_r)^2 \subseteq \cdots \subseteq (\underline{x}^{[n+1]})M: (x_1 \ldots x_r)^n \subseteq \cdots. $$
 By the Noetherianess 
$$(\underline{x})_M^{\lim} = (\underline{x}^{[s+1]})M: (x_1 \ldots x_r)^s$$ for some $s \ge 1$.

\begin{remark}\label{R2.2}
\begin{enumerate}[{(i)}]
\item When $(R, \fm)$ is a local ring and $x_1, \ldots, x_r \in \fm$, it is well-known that $(\underline{x})^{\lim}_M = (\underline{x})M$ if and only if $\underline{x}$ is an $M$-sequence.

\item The notion of limit closure appears naturally when we consider local cohomology as the limit of Koszul cohomology. For a sequence $\underline{x} = x_1, \ldots, x_r$. We have a direct system
$\{M/(\underline{x}^{[n]})M\}_{n \geq 1}$ given by the determinantal
maps
$$(x_1 \ldots x_r)^{m-n}\, :\, M/(\underline{x}^{[n]})M {\longrightarrow} M/(\underline{x}^{[m]})M$$
for $1 \leq n \leq m$. Then the kernel of the canonical map
$$M/(\underline{x})M \to \lim_{\longrightarrow}M/(\underline{x}^{[n]})M \cong H^r_{(\underline{x})}(M)$$
is $(\underline{x})_M^{\lim}/(\underline{x})M$, where
$H^i_{(\underline{x})}(M)$ is the $i$-th local cohomology of $M$
with support in $(\underline{x})$.  We obtain the  induced direct
system $\{M/(\underline{x}^{[n]})_M^{\lim}\}_{n \geq 1}$ with injective maps and
$$\lim_{\longrightarrow}M/(\underline{x}^{[n]})_M^{\lim} \cong H^r_{(\underline{x})}(M).$$
Therefore we can consider
$M/(\underline{x}^{[n]})_M^{\lim}$ as a
submodule of $H^r_{(\underline{x})}(M)$. Hence
$$\mathrm{Ann}(H^r_{(\underline{x})}(M)) = \cap_{n\geq 1}\mathrm{Ann}(M/(\underline{x}^{[n]})_M^{\lim}).$$
In particular, $\mathrm{Ann}(H^r_{(\underline{x})}(R)) = \cap_{n\geq
1}(\underline{x}^{[n]})^{\lim}$.

\item Let $(R, \fm)$ be a local ring and $\underline{x}$ a system of parameters of $R$. The Hochster monomial conjecture is equivalent to say that $(\underline{x})^{\lim} \subseteq \frak m$ for all systems of parameters $\underline{x}$. By Grothendieck's non-vanishing theorem we have $H^t_{\fm}(R) \neq 0$, here $t = \dim R$. According to (ii), for each system of parameters $\underline{x} = x_1, \ldots, x_t$ there exists a positive integer $n_0$, depending on $\underline{x}$), such that $(\underline{x}^{[n]})^{\lim} \subseteq \fm$ for all $n \ge n_0$.
\end{enumerate}
\end{remark}

Let $\underline{y} = y_1, \ldots, y_r$ be another sequence of
elements such that $(\underline{y}) \subseteq (\underline{x})$. Then
there exists a matrix $A = (a_{ij})$, $a_{ij} \in R$, $1 \leq i,j
\leq r$ such that $y_i= \sum_{j=1}^n a_{ij}x_j$, it means
$\mathbf{y} = A \mathbf{x}$, where $\mathbf{x}$ (res. $\mathbf{y}$)
denotes the column vector with entries $x_1, \ldots, x_r$ (res. $y_1, \ldots, y_r$). Following \cite{FH}, we abbreviate it by writing
$(\underline{y}) \overset{A}{\subseteq} (\underline{x})$. It easily follows from Crammer's rule that
$\det (A) (\underline{x}) \subseteq (\underline{y})$. Therefore, we
obtain a canonical map
$$\det (A)  \, :\, M/(\underline{x})M \to M/(\underline{y})M, \quad  m + (\underline{x})M
\mapsto \det (A)  m + (\underline{y})M .$$ By \cite[5.1.15]{S} we also
have that $\det (A)  (\underline{x})_M^{\lim} \subseteq
(\underline{y})_M^{\lim}$. Hence we obtain a homomorphism
$$\det (A)  \, :\, M/(\underline{x})_M^{\lim} \to M/(\underline{y})_M^{\lim},$$
which is independent of the choice of the matrix $A$. The map $\det (A) $ is called a {\it determinantal map}.

\begin{remark}\label{R2.3}
Let $(\underline{y})
\overset{A}{\subseteq} (\underline{x})$ are sequences such that
$\sqrt{(\underline{x})} = \sqrt{(\underline{y})}$ i.e.
$(\underline{x}^{[n]}) \overset{B}{\subseteq} (\underline{y})$ for
some $B$ and $n$. Then the determinantal map $\det (A)  :
M/(\underline{x})_M^{\lim} \to M/(\underline{y})_M^{\lim}$ is
injective (cf. \cite[Lemma 3.1]{CHL}). Therefore
$(\underline{x})_M^{\lim} = (\underline{y})_M^{\lim}:_M \det (A) $. Hence
$(\underline{y})_M^{\lim} \subseteq (\underline{x})_M^{\lim}$.
\end{remark}

The following is a slight generalization of \cite[Proposition
2]{Ho}, and \cite[Theorem 3.3]{CHL}.
\begin{proposition}
Let $M$ be a finitely generated $R$-modules of dimension $d$. Then
there exists a positive integer $n$ such that every system of
parameters $\underline{y} = y_1, \ldots, y_d$ contained in $\frak{m}^n$
satisfies the  monomial property  i.e.
$(\underline{y})_M^{\lim} \neq M$.
\end{proposition}
\begin{proof}
Without any loss of generality we may assume that $\mathrm{Ann}\, M
= 0$. Then by Remark \ref{R2.2} (iii) we  can choose a system of parameters $\underline{x} = x_1, \ldots,
x_d$ satisfies the  monomial property. The positive integer $n$ such
that $\frak{m}^n \subseteq (\underline{x} )$ satisfies our
requirement by Remark \ref{R2.3}.
\end{proof}

\begin{lemma}\label{L2.5}
Let $M$ be a finitely generated $R$-modules of dimension $d$, and
$\underline{x} = x_1, \ldots, x_r$ a sequence of elements in $R$. Then
the following assertions hold true.
\begin{enumerate}[{(i)}]\rm
\item {\it $(\underline{x})_M^{\lim} = M$ if $H^r_{(\underline{x})}(M) =
0$. In particular, if $r>d$ then $ (\underline{x})_M^{\lim} = M$.}
\item {\it If $r = d$ then $M/(\underline{x})_M^{\lim}$ has finite
length.}
\end{enumerate}
\end{lemma}
\begin{proof} Note first that  $M/(\underline{x})_M^{\lim}$ is a
submodule of $H^r_{(\underline{x})}(M)$ by Remark \ref{R2.2} (ii). So the statement
(i) is clear. \\
(ii) If  $r=d$,  $H^d_{(\underline{x})}(M)$ is an
Artinian module (cf. \cite[Excercise 7.1.7]{BS98}), and hence
$M/(\underline{x})_M^{\lim}$ is both Noetherian and Artinian. Thus $\ell(M/(\underline{x})_M^{\lim}) < \infty$.
\end{proof}

\begin{corollary}\label{C2.6}
Let $\underline{x} = x_1, \ldots, x_r$ be a sequence of elements, and $N$ a submodule of $M$ such that $\dim N < r$. Then $N
\subseteq (\underline{x})_M^{\lim}$.
\end{corollary}
\begin{proof} By Lemma \ref{L2.5} we have $(\underline{x})_N^{\lim} =
N$. Hence
\begin{eqnarray*}
(\underline{x})_M^{\lim} &=& \bigcup_{n>0}\big((\underline{x}^{[n+1]})M:_M x_1^n \ldots x_r^n \big)\\
& \supseteq & \bigcup_{n>0}\big((\underline{x}^{[n+1]})N:_N
x_1^n \ldots x_r^n \big)\\
& = & (\underline{x})_N^{\lim} = N.
\end{eqnarray*}
\end{proof}

The following is very useful in the sequel.
\begin{proposition}\label{P2.7}
Let $\underline{x} = x_1, \ldots, x_r$ be a sequence of elements, and $N$ a submodule of $M$ such that $N \subseteq
\cap_{n>0}(\underline{x}^{[n]})_M^{\lim}$. Set $\overline{M} = M/N$.
Then we have
$(\underline{x})_{\overline{M}}^{\lim}=(\underline{x})_M^{\lim}/N$.
\end{proposition}
\begin{proof}
It is sufficient to prove that
$$(\underline{x})_M^{\lim}=
\bigcup_{m>0}\big(((\underline{x}^{[m+1]})M+N):_M x_1^m\ldots x_r^m
\big).$$ The left hand side is clearly contained in the right hand side.  For the converse, by the Noetherianess there exists a positive integer $s$ such
that
$$\bigcup_{m>0}\big(((\underline{x}^{[m+1]})M+N):_M x_1^m \ldots x_r^m
\big) = ((\underline{x}^{[s+1]})M+N):_M x_1^s \ldots x_r^s.$$
On the other hand $ N \subseteq (\underline{x}^{[s+1]})_M^{\lim}$, so there exists a positive
integer $k$ such that
$$N \subseteq (x_1^{(k+1)(s+1)}, \ldots,
x_r^{(k+1)(s+1)})M:_M x_1^{k(s+1)} \ldots x_r^{k(s+1)}.$$ Therefore
$$(x_1^{s+1}, \ldots,x_r^{s+1})M+N \subseteq (x_1^{(k+1)(s+1)}, \ldots,
x_r^{(k+1)(s+1)})M:_M x_1^{k(s+1)} \ldots x_r^{k(s+1)}.$$
Thus
\begin{eqnarray*}
(\underline{x})_M^{\lim} &\supseteq&(x_1^{(k+1)(s+1)}, \ldots,
x_r^{(k+1)(s+1)})M:_M
x_1^{k(s+1)+s} \ldots x_r^{k(s+1)+s}\\
&=& \big((x_1^{(k+1)(s+1)}, \ldots, x_r^{(k+1)(s+1)})M:_M
x_1^{k(s+1)} \ldots x_r^{k(s+1)}\big):_M x_1^s \ldots x_r^s\\
&\supseteq& ((\underline{x}^{[s+1]})M+N) :_M x_1^s \ldots x_r^s \\
&=& \bigcup_{m>0}\big(((\underline{x}^{[m+1]})M+N):_M x_1^m  \ldots x_r^m
\big).
\end{eqnarray*}
The proof is complete.
\end{proof}

\begin{proposition} \label{P3.3}
Let $M$ be a finitely generated $R$-module of dimension $d$ and  $\underline{x} = x_1, \ldots, x_d$ a sequence of elements of $R$. Then the following conditions are equivalent
\begin{enumerate}[{(i)}]\rm
 \item $H^d_{(\underline{x})}(M) = 0$.
\item $\cap_{n\geq 1}(\underline{x}^{[n]})_M^{\lim} = M$.
\item $\dim R/\mathrm{Ann}\,H^d_{(\underline{x})}(M) < d$.
\end{enumerate}
\end{proposition}
\begin{proof} (i) $\Leftrightarrow$ (ii) follows from  Remark \ref{R2.2} (ii).\\
(i) $\Rightarrow$ (iii) is trivial.\\
(iii) $\Rightarrow$ (ii) Set $M' = M/\cap_{n\geq
1}(\underline{x}^{[n]})_M^{\lim}$. By Proposition \ref{P2.7} we have $\cap_{n\geq
1}(\underline{x}^{[n]})_{M'}^{\lim} = 0$. By Remark \ref{R2.2} (ii) again we have
$$\mathrm{Ann}\,H^d_{(\underline{x})}(M) = \cap_{n \ge 1} \mathrm{Ann} \big(M/(\underline{x}^{[n]})_M^{\lim} \big) = \mathrm{Ann} M'.$$
Thus $\dim R/\mathrm{Ann}\,H^d_{(\underline{x})}(M) = \dim M'$. So we have $\dim M' <
d$. It follows by Lemma \ref{L2.5} that $\cap_{n\geq
1}(\underline{x'}^{[n]})_{M'}^{\lim} = M'$. Therefore $M'=0$ and hence $\cap_{n\geq 1}(\underline{x}^{[n]})_M^{\lim} = M$.
\end{proof}

\section{Intersection of limit closures}
The aim of this section is to prove Theorem \ref{thmA}. Let $(R, \fm)$ be a Noetherian local ring and $M$ a finitely generated $R$-module of dimension $d$. Recall that the {\it unmixed component} $U_M(0)$ of $M$  is a submodule defined by
$$U_M(0) =
\bigcap_{\frak p \in \mathrm{Assh}M}N(\frak p),$$ where $0 = \cap_{\frak p \in \mathrm{Ass}\,M}N(\frak
p)$ is a reduced primary decomposition of the zero submodule of $M$ (see \cite{CN}).
We need some auxiliary results from \cite{C,CM} to prove Theorem \ref{thmA}. Let
$\underline{x} = x_1, . . . ,x_d$ be a system of parameters of $M$. Then we can consider
 the differences $$I_{M,\underline{x}}(n) =
\ell(M/(\underline{x}^{[n]})M) - e(\underline{x}^{[n]};M),\, \mathrm { and }$$
$$J_{M,\underline{x}}(n) = e(\underline{x}^{[n]};M) -
\ell(M/(\underline{x}^{[n]})_M^{\lim})\ \ \ \ \ \ $$
as functions in $n$, where $e(\underline{x};M)$ is the Serre multiplicity of $M$ with respect
to the sequence $\underline{x}$. In general, these functions are not polynomials in $n$ (see \cite{CMN}),
but they are bounded above by polynomials. Moreover,  we  have
\begin{theorem}[see \cite{C,CM, MQ17}]\label{C} With the notations as above,
the both functions $I_{M,\underline{x}}(n)$ and
$J_{M,\underline{x}}(n)$ are non-negative increasing, and the least
degrees of polynomials in $n$ bounding above these functions are
independent of the choice of $\underline{x}$. Moreover, if we denote
by $p(M)$ and $pf(M)$ for these least degrees with respect to
$I_{M,\underline{x}}(n)$ and $J_{M,\underline{x}}(n)$ respectively,
then $p(M)\leq d-1$ and $pf(M)\leq d-2$.
\end{theorem}
Now we are able to prove Theorem \ref{thmA}.
\begin{proof}[Proof of Theorem \ref{thmA}.] We set $N =
\cap_{n>0}(\underline{x}^{[n]})_M^{\lim}$. Recalling that $\dim U_M(0) < d$, we
have $U_M(0) \subseteq N$ by Corollary \ref{C2.6}. Put $\overline{M} = M/U_M(0)$ and $M' =
M/N$. By Proposition \ref{P2.7} we have
$$\ell(M/(\underline{x}^{[n]})_M^{\lim}) = \ell(M'/(\underline{x}^{[n]})_{M'}^{\lim}) = \ell(\overline{M}/(\underline{x}^{[n]})_{\overline{M}}^{\lim})$$
for all $n \ge 1$. Then by Theorem \ref{C}, there are polynomials $f(n)$ of degree at most $d-1$ and $ g(n)$ of
degree at most $ d-2$ such that
$$\ell(M/(\underline{x}^{[n]})_M^{\lim}) = \ell(M'/(\underline{x}^{[n]})_{M'}^{\lim}) \leq
\ell(M'/(\underline{x}^{[n]})M') \leq n^d e(\underline{x};M') +
f(n),$$ and
$$\ell(M/(\underline{x}^{[n]})_M^{\lim}) = \ell(\overline{M}/(\underline{x}^{[n]})_{\overline{M}}^{\lim})
\geq n^d e(\underline{x};\overline{M}) - g(n).$$
Therefore
$$f(n) + g(n) \geq n^d (e(\underline{x},\overline{M}) - e(\underline{x},M'))$$ for all $n>0$.
It follows
that $e(\underline{x};N/U_M(0)) = e(\underline{x};\overline{M}) -
e(\underline{x};M') = 0$. Hence $\dim\, N <d$, and so  $N
= U_M(0)$, since $U_M(0)$ is the
largest submodule of $M$ with the dimension less than $d$. The proof is complete.
\end{proof}
The following result is an immediate consequence of Theorem \ref{thmA}.
\begin{corollary} \label{C4.3}
$\bigcap_{\underline x} (\underline{x})_M^{\lim} = U_M(0) $,
where $\underline{x}$ runs through the
set of all systems of parameters of $M$.
\end{corollary}

\begin{corollary} \label{C3.2}
Let $M$ be a finitely generated $R$-module of dimension $d$.
$$\mathrm{Ann}\,(H^d_{\mathfrak{m}}(M)) = \mathrm{Ann}\,(M/U_M(0)) = \{r \in R \mid \dim M/(0:_Mr) < d\}.$$
In particular, $H^d_{\mathfrak{m}}(M) \neq 0$.
\end{corollary}
\begin{proof}
Let $\underline{x} = x_1, \ldots, x_d$ be a system of parameter of $M$.
By Remark \ref{R2.2}, we may consider
$M/(\underline{x}^{[n]})_M^{\lim}$ as a submodule of
$H^d_{\mathfrak{m}}(M) = H^d_{(\underline{x})}(M)$ for all $n \geq
1$.
\begin{eqnarray*}
\mathrm{Ann}\,(H^d_{\mathfrak{m}}(M)) &=& \{r \in R \mid rM
\subseteq M/(\underline{x}^{[n]})_M^{\lim}, \forall \, n \geq 1\}\\
&=& \{r \in R \mid rM \subseteq U_M(0)\}\\
& =&  \mathrm{Ann}\,(M/U_M(0))\\
&=& \{r \in R \mid \dim (rM) < d\}\\
&=& \{r \in R \mid \dim M/(0:_Mr) < d\}.
\end{eqnarray*}
The last assertion follows from the first.
\end{proof}
The following was proved by Grothendieck \cite[ Proposition 6.6
(7)]{Gr}.
\begin{corollary} Let $(R,\frak m)$ be a complete local ring, and $M$
a finitely generated $R$-module of dimension $d$. Let $T^d(M) =
\mathrm{Hom}_R(H^d_{\mathfrak{m}}(M),E(R/\frak m))$, where
$E(R/\frak m)$ is the injective envelope of $R/\frak m$. Then
$\mathrm{Ann}\,(T^d(M)) = U_{R/\mathrm{Ann} M}(0)$.
\end{corollary}
\begin{proof} We may assume that $\mathrm{Ann} M = 0$. Therefore $\mathrm{Assh}M = \mathrm{Assh} R$. By duality we have $\mathrm{Ann} T^d(M) = \mathrm{Ann} H^d_{\fm}(M)$. So by Corollary
\ref{C3.2} we need only to show that $\mathrm{Ann}\,(M/U_M(0)) =
U_{R}(0)$ which is equivalent to
$$\{r \in R \mid \dim rM < d\} = \{r \in R \mid \dim rR < d\}.$$
Since $M$ is a faithfully $R$-module then the following homomorphism
$$R \to M^k, x \mapsto (xm_1, \ldots,xm_k)$$
is injective, where
$m_1, \ldots,m_k$ are generators of $M$. We also have a natural projective
homomorphism $R^k \to M$. Thus for all $\frak p \in \mathrm{Spec}(R)$ and for all $r \in R$ we have $rM_\frak p = 0$ iff $r R_{\frak p} = 0$. Therefore
\begin{eqnarray*}
\{r \in R \mid \dim rM < d\} &=& \{r \in R \mid rM_{\frak p}=0,
\forall\, \frak p \in \mathrm{Assh}M\}\\
&=& \{r \in R \mid rR_{\frak p}=0, \forall\, \frak p \in
\mathrm{Assh}R\}\\
&=& \{r \in R \mid \dim rR < d\}.
\end{eqnarray*}
The proof is complete.
\end{proof}

\begin{corollary}\label{C5.4}
Let $(R, \frak{m})$ be a complete local ring, and $M$ a finitely
generated $R$-module of dimension $d$. Let $\underline{x} = x_1, \ldots, x_d$ be a sequence of elements. Let $\cap_{\frak p \in
\mathrm{Ass}\,M}N(\frak p) = 0$ be a reduced primary decomposition
of $(0)$ in $M$. Then
$$\bigcap_{n\geq 1}
(\underline{x}^{[n]})_M^{\lim} = \bigcap_{\frak p \in J}N(\frak p),$$
where $J = \{\frak p \in \mathrm{Ass}M \mid \underline{x}\, \text{
is a system of parameters of}\, R/\frak p\}$.
\end{corollary}
\begin{proof}
Set $N = \cap_{\frak p\in J}N(\frak p)$.  Then $\mathrm{Ass}\,N =
\mathrm{Ass}\,M \setminus J$. By the Lichtenbaum-Hartshorne vanishing Theorem
we have $H^d_{(\underline{x})}(N) = 0$. Therefore
$\cap_{n\geq 1}(\underline{x}^{[n]})^{\lim}_M \supseteq \cap_{n\geq
1}(\underline{x}^{[n]})^{\lim}_N = N$ by Proposition \ref{P3.3}. Set $M' =
M/N$, we have $\Ass M' = J \subseteq \mathrm{Assh}M$. Then $\underline{x}$ is a system of parameters of $M'$ and
$U_{M'}(0)=0$. Hence $\cap_{n\geq 1}(\underline{x}^{[n]})_{M'}^{\lim}
= 0$ by Theorem \ref{thmA}. Thus $\cap_{n\geq
1}(\underline{x}^{[n]})^{\lim}_M = N$ by Proposition \ref{P2.7}.
\end{proof}

We prove a non-vanishing result of local cohomology.

\begin{corollary} Let $M$ be a finitely generated $R$-module of dimension $d$ and $x_1, \ldots, x_r$ be a part of system of parameters. Then $H^r_{(x_1, \ldots, x_r)}(M) \neq 0$.
\end{corollary}
\begin{proof} Extend $x_1, \ldots, x_r$ to a full system of parameters $\underline{x} = x_1, \ldots, x_d$. For all $n \ge 1$ we have $(x_1^n, \ldots, x_r^n)^{\lim}_M \subseteq (\underline{x}^{[n]})^{\lim}_M$. Therefore
  $$\bigcap_{n \ge 1}(x_1^n, \ldots, x_r^n)^{\lim}_M \subseteq U_M(0)$$
by Theorem \ref{thmA}. By Remark \ref{R2.2} (ii) we have
$$\Ann H^r_{(x_1, \ldots, x_r)}(M) = \bigcap_{n \ge 1} \Ann M/(x_1^n, \ldots, x_r^n)^{\lim}_M \subseteq \Ann M/U_M(0).$$
Thus $H^r_{(x_1, \ldots, x_r)}(M) \neq 0$. The proof is complete.
\end{proof}

\section{Systems of parameters}

In this section, we prove a characterization of systems of parameters in terms of injectivity of the determinatal maps. Our results improve known results of Dutta and Roberts in \cite{DR} and of Fouli and Huneke in \cite{FH}. Let $\underline{x} = x_1, \ldots, x_t$  be a system of parameters of $R$, and $\underline{y} =
y_1, \ldots, y_t$ a sequence of elements such that $(\underline{y})
\overset{A}{\subseteq} (\underline{x})$.
 Dutta and Roberts  proved in \cite{DR} that  if $R$ is a Cohen-Macaulay ring, then $\underline{y}$ is a system of
parameters if and only if the determinantal map $\det A: R/(\underline{x}) \to
R/(\underline{y})$  is injective. The following result is a slight generalization of Dutta-Roberts's result.
\begin{theorem}\label{T2.9}
Let $R$ be a local ring such that $R/U_R(0)$ is Cohen-Macaulay. Let
$\underline{x}$ be a system of parameters and $\underline{y}$ a
sequence of elements in $R$ such that $(\underline{y})
\overset{A}{\subseteq} (\underline{x})$. Then $\underline{y}$ is a
system of parameters of $R/U_R(0)$ if and only if the determinantal
map $\det A: R/(\underline{x})^{\lim} \to R/(\underline{y})^{\lim}$
 is injective.

\end{theorem}
\begin{proof} By Corollary \ref{C2.6} we have that $U_R(0)
\subseteq (\underline{x}^{[s]})^{\lim}$ for all $s \geq 1$. Set
$R' = R/U_R(0)$, and we consider $R'$ as an $R$-module. By
Proposition \ref{P2.7} we have
$$(\underline{x})_{R'}^{\lim} = (\underline{x})^{\lim}/U_R(0); (\underline{y})_{R'}^{\lim} = (\underline{y})^{\lim}/U_R(0).$$
Therefore the deteminatal maps
$$\det A: R/(\underline{x})^{\lim} \to R/(\underline{y})^{\lim}$$
and
$$\det A: R'/(\underline{x})_{R'}^{\lim} \to R'/(\underline{y})_{R'}^{\lim}$$
coincide. Notice that $R'$ is Cohen-Macaulay. Then the conclusion follows from the module-version of Dutta-Roberts' theorem.
\end{proof}
As a consequence of Theorem \ref{T2.9} we obtain the following result of Fouli and Huneke \cite[Theorem 4.4]{FH}.
\begin{corollary}
Let $R$ be a $1$-dimensional Noetherian local ring. Let $x$ be a
parameter element, and let $y = ux$. Then $y$ is a parameter element if and only if
the map $R/(x)^{\lim} \overset{u}{\longrightarrow} R/(y)^{\lim}$ is
injective.
\end{corollary}
\begin{proof}
Since $\dim\,R=1$ we have that $U_R(0) = H^0_{\frak m}(R)$ and
$\overline{R} = R/U_R(0)$ is Cohen-Macaulay. Moreover, $y$ is a
parameter element of $R$ if and only if $y$ is also a parameter
element of $\overline{R}$. Hence the assertion follows from Theorem
\ref{T2.9}.
\end{proof}

In higher dimension, we need the condition that $R$ is equidimensional to claim that if a sequence $\underline{x} = x_1, \ldots ,x_t$ is a system of parameters of $R/U_R(0)$, then it is a system of parameters of $R$. We need the following in the sequel.
\begin{lemma}\label{L5.3}
Let $(R,\frak m)$ be a catenary equidimensional local ring of
dimension $t$, $\underline{x} = x_1, \ldots ,x_t$ a sequence of
elements of $R$. Then the following statements are equivalent:
\begin{enumerate}[{(i)}]\rm
    \item {\it $\underline{x}$ is a system of parameters of $R$;}
    \item {\it $\underline{x}$ is a system of parameters of
    $\widehat{R}/U_{\widehat{R}}(0)$.}
\end{enumerate}
\end{lemma}
\begin{proof}
(i) $\Rightarrow$ (ii) is clear.\\
(ii) $\Rightarrow$ (i) It is easily seen that $\underline{x}$ is a
system of parameters of $\widehat{R}/U_{\widehat{R}}(0)$ if and only if
$\underline{x}$ is a system of parameters of $\widehat{R}/\frak P$
for all $\frak P \in \mathrm{Assh} \widehat{R}$. We will
prove by induction on $t$ that if $\underline{x}$ is a system of
parameters of $\widehat{R}/\frak P$ for all $\frak P \in
\mathrm{Assh} \widehat{R}$, then $\underline{x}$ is a system of
parameters of $R$. The case $t=1$ is trivial. Suppose that $t>1$. We
choose an arbitrary prime ideal $\frak p \in \mathrm{Assh} R$, then
there exists $\frak P \in \mathrm{Assh}\widehat{R}$ such that $\frak
P \cap R = \frak p$. Since $\underline{x}$ is a system of parameters
of $\widehat{R}/\frak P$ we have $x_1 \notin \frak P$, and so $x_1
\notin \frak p$. Thus $R/(x_1)$ is a catenary equidimensional local
ring of dimension $t-1$. We will show that $\underline{x}' = x_2, \ldots ,x_t$ is a system of parameter of $\widehat{R}/\frak P'$ for all
$\frak P' \in \mathrm{Assh} \widehat{R}/(x_1)\widehat{R}$ . If
$\frak P' \in \mathrm{min}(\widehat{R})$, then $\frak P' \cap R =
\frak p \in \mathrm{Assh}R$ since $R$ is equdimensional. Therefore $x_1 \notin \frak P'$, it is
a contradiction. Thus there exists $\frak P \in \mathrm{Assh}
\widehat{R}$ such that $\frak P \subsetneq \frak P'$, note that $\dim \widehat{R}/\frak P' = d-1$. Because $\underline{x} = x_1, \ldots, x_t$ is a system of parameters of $\widehat{R}/ \frak P$ we have
$\underline{x}' = x_2, \ldots ,x_t$ is a system of parameter of
$\widehat{R}/\frak P'$. By the inductive hypothesis $\underline{x}' = x_2, \ldots ,x_t$ is a system of parameters of $R/(x_1)$. Thus $\underline{x}$ is a system of parameters of $R$ as required.
\end{proof}

The following is the main result of this section.
\begin{theorem} \label{T5.4}
Let $(R,\frak m)$ be a catenary local ring of dimension $t$. There
exists a positive integer $\ell$, which depends only on $R$, with property: whenever $\underline{x} =
x_1, \ldots ,x_t$ is a system of parameters of $R/U_R(0)$ with
$(\underline{x}) \subseteq \frak m^\ell$ and $\underline{y}
=y_1, \ldots,y_t$ a sequence of elements such that $(\underline{y})
\overset{A}{\subseteq} (\underline{x})$ the following statements are
equivalent:
\begin{enumerate}[{(i)}]\rm
    \item {\it $\underline{y}$ forms a system of parameters of $R/U_R(0)$;}
    \item {\it The determinantal map $R/(\underline{x})^{\lim} \xrightarrow{\det
A} R/(\underline{y})^{\lim}$ is injective.}
\end{enumerate}
\end{theorem}

\begin{proof} (i) $\Rightarrow$ (ii) follows from Remark \ref{R2.3} and the same argument as the proof of Theorem \ref{T2.9}.\\
(ii) $\Rightarrow$ (i). We first show that the assertion in the
case $R$ is complete. By Corollary \ref{C2.6} and Proposition
\ref{P2.7} we may assume henceforth that $U_R(0) = 0$ and hence $R$ is
equidimensional. Assume that $\mathrm {Ass }R = \mathrm {Assh }R = \{\frak p_1, \ldots,\frak
p_n\}$ and $\cap_{\frak p _i\in \mathrm{Ass}\,R}N(\frak p_i) = 0$ is
a reduced primary decomposition of $(0)$. For $1\leq i\leq n$ we set
$$L_i = \bigcap_{j \neq i, \frak p_j \in
\mathrm{Ass}R}N(\frak p_j).$$
Let $\underline{z} = z_1, \ldots ,z_t$
is a system of parameters of $R$.  By Theorem \ref{thmA} we have
$\cap_{n \geq 1}(\underline{z}^{[n]})^{\lim} = 0$. Then there is a
positive integer $\ell_1$ such that $L_i \nsubseteq
(\underline{z}^{[\ell_1]})^{\lim}$ for all $ i= 1, \ldots, n$. Let $\ell$ be a positive integer such that $\frak m^{\ell} \subseteq
(\underline{z}^{[\ell_1]})$.

Suppose we have $\underline{x} = x_1, \ldots ,x_t$
and $\underline{y} =y_1, \ldots,y_t$ are sequences of elements contained
in $\frak m^{\ell}$ such that $(\underline{y})
\overset{A}{\subseteq} (\underline{x})$, $\underline{x}$ is a
system of parameters of $R$ and the determinantal map $R/(\underline{x})^{\lim}
\xrightarrow{\det
A} R/(\underline{y})^{\lim}$ is injective. Assume  $\underline{y}$ is not a system of parameter of $R$. By relabeling (if necessarily) we can assume henceforth that $\underline{y} =y_1, \ldots,y_t$ is not a system of parameters of
$R/\frak p_1$. By Corollary \ref{C5.4} we have $0 \neq L_1 \subseteq
(\underline{y})^{\lim}$.  On the other hand, it follows from Remark
\ref{R2.3}  that $(\underline{x})^{\lim} \subseteq
(\underline{z}^{[\ell_1]})^{\lim}$. Hence $L_i \nsubseteq
(\underline{x})^{\lim}$ for all $i = 1, \ldots, n$. Thus there is $0 \neq u \in
(\underline{y})^{\lim} \setminus (\underline{x})^{\lim}$. Therefore the determinantal map $R/(\underline{x})^{\lim}
\xrightarrow{\det
A} R/(\underline{y})^{\lim}$ maps $u + (\underline{x})^{\lim}$ to $0$. So it is not
injective, a contradiction. Hence $\underline{y} =y_1, \ldots,y_t$ is a system of parameters of $R$.\\
The assertion in general case now follows from Lemma \ref{L5.3} since $R/U_R(0)$ is catenary and equidimensional. The proof is complete.
\end{proof}

\begin{proof}[Proof of Theorem \ref{thmC}] It immediately follows from Theorem \ref{T5.4} and the fact $\underline{x}$ is a system of parameters of $R$ if and only if $\underline{x}$ is a system of parameters of $R/U_R(0)$ provided $R$ is equidimensional.
\end{proof}

Theorem \ref{thmC} was proved by Fouli and
Huneke for any equidimensional local ring (cf. \cite[Corollary 5.4]{FH}). However, in fact, they proved this result under additional assumption that $R$ is complete. Thus our result is a generalization of their one. We will show that the catenary condition is essential.

\begin{example}[see, \cite{N}, Example 2, pp 203--205]\rm Let $K$ be a
field, and $K[[X]]$ a formal power series ring. Let $Z = \sum_{i\geq
1}a_ix^i$ be an algebraically independent element over $K(X)$. Set
$Z_j =(Z - \sum_{k<j}a_kX^k)/X^{j-1}$. Furthermore let $Y$ be an
algebraically independent element over $K[X, Z]$. Let $R_1 = K[X,
Z_1, \ldots,Z_j,\ldots]$ and set $R_2 = R_1[Y]$. Let $\frak n_1 = (X, Y)$,
$\frak n_2 = (X-1, Z, Y)$ are maximal ideals of $R_2$ with
$\mathrm{ht}(\frak n_1) = 2$ and $\mathrm{ht}(\frak n_2) = 3$. Let
$S$ be the intersection of complements of $\frak n_1$ and $\frak
n_2$ in $R_2$ and set $R' = (R_2)_S$. Then $R'$ is Noetherian. Let
$\frak m$ be the Jacobson radical of $R'$ and set $R = K+\frak m$.
We have $(R,\frak m)$ is a local domain of dimension $3$. However
$R$ is non-catenary since $0 \subset \frak q = XR' \cap R \subset
\frak m$ is a maximal chain of prime ideals in $R$. Since $(R, \fm)$ is a local ring of dimension three, $\fq$ is generated by three elements up to radical.
\end{example}
\begin{proposition} Let $R, \frak m, \frak q$ as in the previous Example, and let $\underline{x} = x_1, x_2, x_3$ be any system of parameters of $R$. Then
there exists $\underline{y} = y_1,y_2,y_3$ such that
$(\underline{y}) \overset{A}{\subseteq} (\underline{x})$ and the
determinantal map $R/(\underline{x})^{\lim} \xrightarrow{\det
A} R/(\underline{y})^{\lim}$ is injective but
$\underline{y}$ is not a system of parameters.
\end{proposition}
\begin{proof} As above, $\fq = \sqrt{(z_1, z_2, z_3)}$ for some  $z_1, z_2, z_3$. We may choose the sequence $\underline{y} = y_1, y_2,y_3$ with $y_i = z_i^k$, $1\le i \le 3$, for large enough $k$ such that
$(\underline{y}) \overset{A}{\subseteq} (\underline{x})$. It is clear that
$\underline{y}$ is not a system of parameters of $R$. We now show that the
determinantal map
$$R/(\underline{x})^{\lim} \xrightarrow{\det
A} R/(\underline{y})^{\lim}$$
is injective. It is equivalent to prove that the
determinantal map
$$\widehat{R}/(\underline{x})_{\widehat{R}}^{\lim} \xrightarrow{\det
A} \widehat{R}/(\underline{y})_{\widehat{R}}^{\lim}$$ is injective.
By
Theorem \ref{T5.4} it is sufficient to prove that $\underline{y}$ is
a system of parameters of $\widehat{R}/\frak P$ for any $\frak P \in
\mathrm{Assh}\widehat{R}$, that is $\dim \widehat{R} / \frak P = 3$. Indeed, let $\frak q \widehat{R} = \frak Q_1 \cap
\cdots \cap \frak Q_r$, where $\frak Q_i$ is a $\frak P_i$-primary,
be a reduced primary decomposition of $\frak q \widehat{R}$. Since
$\dim R/\frak q =1$ we have $\dim \widehat{R}/\frak P_i =1$ and
$\mathrm{ht} (\frak P_i/\frak q \widehat{R})=0$ for all $i\leq r$.
Thus by \cite[Theorem 15.1]{Mat} we have $\mathrm{ht}(\frak P_i) =
\mathrm{ht} (\frak q) + \mathrm{ht}(\frak P_i/\frak q\widehat{R}) =
1$ for all $i \leq r$. Moreover, $\widehat{R}$ is
catenary, so $\frak P \nsubseteq \frak P_i$ for all $i \leq r$. Therefore $\dim
\widehat{R}/(\frak q \widehat{R} + \frak P) = 0$. Hence
$\underline{y}$ is a system of parameters of $\widehat{R}/\frak P$
for all $\frak P \in \mathrm{Assh}\widehat{R}$ since
$\sqrt{(\underline{y})}= \frak q$. The proof is complete.
\end{proof}

\section{A characterization of unmixed local rings}

 Unmixed local rings  were introduced by Nagata \cite{N} as follows.
\begin{definition}\rm Let $(R,\frak m)$ be a Noetherian local ring of dimension $t$. Then $R$
is {\it unmixed} if $U_{\widehat{R}}(0)=0$ i.e. $\mathrm{Assh}\,
\widehat{R} = \mathrm{Ass}\, \widehat{R}$, where $\widehat{R}$
denotes the completion of $R$ with respect to the $\frak m$-adic
topology.
\end{definition}
Almost of domains in Commutative Algebra are unmixed. However, in
\cite[Example 2, pp. 203--205]{N} Nagata constructed a domain of
dimension two which is not unmixed. Unmixed local rings were
investigated by several authors (cf. \cite{Ra1,Ra2,Tr81}). Let $\underline{x} = x_1, \ldots, x_t$ be a system of
parameters of $R$. By Krull's intersection theorem we have $\cap_{n
\geq 1}(\underline{x}^{[n]}) = 0$. It means that the topology
defined by $\{(\underline{x}^{[n]})\}_{n \geq 1}$ is always
Hausdorff. However, the topology defined by
$\{(\underline{x}^{[n]})^{\lim}\}_{n \geq 1}$ may be not Hausdorff.
In fact, $\{(\underline{x}^{[n]})^{\lim}\}_{n \geq 1}$ is a Hausdorff
topology if and only if $U_R(0)=0$ by Theorem \ref{thmA}. The aim of
this section is to give a characterization of unmixed local rings in
terms of the topology defined by $\{(\underline{x}^{[n]})^{\lim}\}_{n
\geq 1}$.
\begin{lemma} Let $\underline{x} = x_1, \ldots, x_t$ and $\underline{y} = y_1, \ldots, y_t$ be systems of parameters of $R$. Then the topology
defined by $\{(\underline{x}^{[n]})^{\lim}\}_{n \geq 1}$ and $\{(\underline{y}^{[n]})^{\lim}\}_{n \geq 1}$ are equivalent.
\end{lemma}
\begin{proof} For each $n \ge 1$ there is an integral $v(n)$ such that $(\underline{x}^{[n]}) \supseteq (\underline{y}^{[v(n)]})$. By Remark \ref{R2.3} we have $(\underline{x}^{[n]})^{\lim} \supseteq (\underline{y}^{[v(n)]})^{\lim}$. Therefore the topology defined by $\{(\underline{y}^{[n]})^{\lim}\}_{n \geq 1}$ is stronger than or equal to the topology defined by $\{(\underline{x}^{[n]})^{\lim}\}_{n \geq 1}$. Symmetrically, we have the converse. The proof is complete.
\end{proof}
By the previous Lemma we can define a topology of $R$ as follows.
\begin{definition}\rm Let $(R, \frak m)$ be a local ring of dimension $t$. We define the {\it limit closure topology} of $R$ the topology defined by $\{(\underline{x}^{[n]})^{\lim}\}_{n \geq 1}$ for some system of parameters $\underline{x} = x_1, \ldots, x_t$.
\end{definition}

The following result, proved by Chevalley (cf. \cite [Lemma 7]{Ch}), plays the key role in our proof the main result of this section.
\begin{lemma}[Chevalley] \label{L4.2}
Let $(R,\frak m)$ be a complete Noetherian local ring, and $\frak
a_1 \supseteq \frak a_2 \supseteq \cdots$ a chain of ideals of $R$
such that $\cap_{n\geq 1}\frak a_n = 0$. Then for each $n$ there
exists an integer $v(n)$ such that $\frak a_{v(n)} \subseteq \frak m^n$. In
other words, the linear topology defined by $\{\frak a_n\}_{n\geq
1}$ is stronger than or equal to the $\frak m$-adic topology.
\end{lemma}
We now prove Theorem \ref{thmD} proposed in the introduction.
\begin{theorem}\label{T4.3}
Let $(R,\frak m)$ be a Noetherian local ring of dimension $t$. Then $R$ is
unmixed if and only if the $\frak m$-adic and limit closure topologies are equivalent.
\end{theorem}
\begin{proof} We note that the $\frak m$-adic topology is
always stronger than or equal to the topology defined by $\{(\underline{x}^{[n]})^{\lim}\}_{n
\geq 1}$  since  $(\underline{x}^{[n]})^{\lim}$ is $\frak
m$-primary for all $n \geq
1$.\\
$(\Rightarrow)$ We assume that $R$ is unmixed. Then by Theorem \ref{thmA} the  topology defined by
$\{(\underline{x}^{[n]})_{\widehat{R}}^{\lim}\}_{n \geq 1}$
is Hausdorff. By Chevalley's theorem, for each $n$ there exists
 an integer $v(n)$ such that $(\underline{x}^{[v(n)]})_{\widehat{R}}^{\lim}
\subseteq \widehat{\frak m}^n$. Thus $(\underline{x}^{[v(n)]})^{\lim}
\subseteq {\frak m}^n$. Therefore the
topology defined by $\{(\underline{x}^{[n]})^{\lim}\}_{n \geq 1}$  is stronger than or
equal to the $\frak
m$-adic topology. Thus they are equivalent.\\
$(\Leftarrow)$ Suppose that $R$ is not unmixed i.e. $U_{\widehat{R}}(0)
\neq 0$. By Krull's intersection theorem, there exists $n_0$ such
that $U_{\widehat{R}}(0) \nsubseteq {\widehat{\frak m}}^{n_0}$. On the other hand, we get by
Theorem \ref{thmA} that
$U_{\widehat{R}}(0) \subseteq
(\underline{x}^{[n]})_{\widehat{R}}^{\lim}$ for all $n \ge 1$. Therefore
$(\underline{x}^{[n]})_{\widehat{R}}^{\lim} \nsubseteq
{\widehat{\frak m}}^{n_0}$ for all $n$. Thus
$(\underline{x}^{[n]})^{\lim} \nsubseteq {{\frak m}}^{n_0}$ for all
$n \ge 1$, so the  topology defined by $\{(\underline{x}^{[n]})^{\lim}\}_{n \geq 1}$  is
not equivalent to the $\frak m$-adic topology.
\end{proof}
\begin{corollary}\label{C4.4}
Let $(R,\frak m)$ be a Noetherian local ring such that $U_R(0)=0$. Suppose that
 the $\frak m$-adic topology is minimal among all Hausdorff
topologies of $R$. Then $R$ is unmixed.
\end{corollary}
\begin{example}\rm By Nagata (cf. \cite[Example 2, pp. 203--205]{N}) we have there exists a local domain $(R, \frak m)$ of dimension two such that $\widehat{R} \cong k[[X,Y,Z]]/((X) \cap (Y,Z)) = k[[x,y,z]]$, where $k$ is a field. Let $a, b$ be a system of parameters of $R$. We have $U_{\widehat{R}}(0) = (x) \notin \widehat{\frak m}^2$. Therefore, $(a^n,b^n)^{\lim}_{\widehat{R}} \nsubseteq \widehat{\frak m}^k$ for all $n \ge 1$ and all $k \ge 2$. Thus $\{(a^n,b^n)^{\lim}\}_{n \ge 1}$ is a Hausdorff topology of $R$ but $(a^n,b^n)^{\lim} \nsubseteq \frak m^k$ for all $n \ge 1$ and all $k \ge 2$.
\end{example}
\section{Limit closure in local rings of dimension two}
By the monomial conjecture (theorem) we have $(\underline{x})^{\lim} \subseteq \fm$ for every system of parameters $\underline{x}$. Recently, Ma, Smirnov and the second author \cite{MQS20} proved under mild condition that $(\underline{x})^{\lim} \subseteq \overline{(\underline{x})}$ the integral closure of $(\underline{x})$. The purpose of the section is to give some explicit descriptions for the limit closure in local rings of dimension two. Let $\underline{x} = x_1, \ldots, x_d$ be a system of parameters of the local ring $(R, \frak m)$. Since $(\underline{x})^{\lim}_{\widehat{R}} = (\underline{x})^{\lim}\widehat{R}$, we shall assume that $(R, \frak m)$ is an image of a Cohen-Macaulay local ring.\\

\noindent {\bf $S_2$-ification.}(cf. \cite{HH94}) Suppose $R$ is an unmixed local ring. We shall say that a ring $S$ is an {\it $S_2$-ification} of $R$ if it lies between $R$ and its total quotient ring, is module-finite over $R$, is $S_2$
as an $R$-module, and has the property that for every element $s \in S - R$, the ideal $D(s)$, defined as $\{r \in R \mid rs \in R\}$, has height
at least two.
\begin{remark}\rm
\begin{enumerate}[{(i)}]
\item If $(R, \frak m)$ is complete, $ R$ has a $S_2$-ification, and it is unique. Moreover, if $\omega$ is a canonical module of $R$, then $S \cong \mathrm{Hom}(\omega, \omega)$.
\item Let $\frak a_i = \mathrm{Ann}H^i_{\frak m}(R)$, $i = 0, \ldots, d$, and $\frak a = \frak a_0 \ldots \frak a_{d-1}$. If $R$ is an image of a Cohen-Macaulay local ring, we have $\dim R/\frak a_i \le i$ and so $\dim R/\frak a \le d-1$ (cf. \cite[Theorem 8.1.1]{BH98}, \cite{CC15}). Moreover if $R$ is unmixed we have $\dim R/\frak a \le d-2$. We have the $S_2$-ification of $R$ is just the ideal transformation $D_{\frak a}(R)$ (see \cite{BS98} for the definition of ideal transformation). Thus the $S_2$-ification of an unmixed local ring exists provided the ring is an image of a Cohen-Macaulay local ring.
\end{enumerate}
\end{remark}

The following appears implicitly in the proof of \cite[Theorem 4.3]{MQ17}. For the sake of completeness
we give a detailed proof.
\begin{theorem}\label{T5.2}
Let $(R, \frak m)$ be an unmixed local ring of dimension $d$ and $\underline{x} = x_1, \ldots, x_d$ a system of parameters of $R$. Let $S$ be the $S_2$-ification of $R$. Then $(\underline{x})^{\lim} = (\underline{x})^{\lim}_{S} \cap R$.
\end{theorem}
\begin{proof}
Consider the exact  sequence
$$0 \to R \to S \to S/R \to 0,$$
with $\dim S/R \le d-2$. Applying the functors local cohomology and Koszul's cohomology we have the following commutative diagram.
\[\divide\dgARROWLENGTH by 2
\begin{diagram}
\node{} \node{} \node{H^{d-1}(\underline{x}; S/R)} \arrow{e}\arrow{s} \node{H^{d-1}_{\frak m}(S/R) = 0}\arrow{s}\\
\node{0}\arrow{e}\node{(\underline{x})^{\lim}/(\underline{x})R} \arrow{e}\arrow{s}\node{H^d(\underline{x};R)}\arrow{e}\arrow{s}
\node{H^d_{\frak m}(R)}\arrow{s}\\
\node{0}\arrow{e}\node{(\underline{x})_S^{\lim}/(\underline{x})S} \arrow{e}
\arrow{e}\node{H^d(\underline{x};S)}\arrow{e}\arrow{s} \node{H^d_{\frak m}(S)}\arrow{s}\\
\node{} \node{} \node{H^{d}(\underline{x}; S/R)} \arrow{e} \node{H^{d}_{\frak m}(S/R) = 0.}
\end{diagram}
\]
Both the second and third rows are exact by Remark \ref{R2.2}. Thus we have the following commutative diagram
\[\divide\dgARROWLENGTH by 2
\begin{diagram}
\node{R/(\underline{x})^{\lim}}\arrow{e,t}{\pi}\arrow{s,l}{\alpha}
\node{H^d_{\frak m}(R)}\arrow{s,l}{\sigma}\\
\node{S/(\underline{x})_S^{\lim}} \arrow{e,t}{\tau}\node{H^d_{\frak m}(S)}
\end{diagram}
\]
with $\pi$ and $\tau$ are injective and $\sigma$ is bijective. The equation $\tau \circ \alpha = \sigma \circ \pi$ implies that the map
$$\alpha: R/(\underline{x})^{\lim} \to S/(\underline{x})_S^{\lim}$$
is injective. Therefore $(\underline{x})^{\lim} = (\underline{x})^{\lim}_{S} \cap R$.
\end{proof}
In the rest of this section we assume that $\dim R = 2$ and $x, y$ is a system of parameters of $R$. Let $U_R(0)$ is the unmixed component of $R$ and $\overline{R} = R/U_R(0)$. By Theorem \ref{thmA} and Proposition \ref{P2.7} we have
$$(x,y)^{\lim} = \bigcup_{n \ge 1} \big( (x^{n+1}, y^{n+1}, U_R(0)) :_R (xy)^n \big)$$
and $\ell(R/(x,y)^{\lim}) = \ell (\overline{R}/(x,y)^{\lim}_{\overline{R}})$. Hence we can reduce to the case of an unmixed ring. In this case we have the $S_2$-ification $S$ of $R$ is Cohen-Macaulay (since $d=2$). Moreover, $H^1_{\frak m}(R)$ is finitely generated (see, \cite{Tr81}) and $S/R \cong H^1_{\frak m}(R)$. We prove the main result of this section as follows.
\begin{proof}[Proof of Theorem \ref{thmE}] (i) follows from Theorem \ref{T5.2} and the Cohen-Macaulayness of $S$.\\
(ii) By (i) the short exact sequence
$$0 \to R \to S \to H^1_{\frak m}(R) \to 0$$
induces the short exact sequence
$$0 \to R/(x,y)^{\lim} \to S/(x,y)S \to H^1_{\frak m}(R)/(x,y)H^1_{\frak m}(R) \to 0.$$
Therefore
$\ell (R/(x,y)^{\lim}) = \ell(S/(x,y)S) - \ell (H^1_{\frak m}(R)/(x,y)H^1_{\frak m}(R))$.
Since $S$ is Cohen-Macaulay we have $\ell(S/(x,y)S) = e(x,y; S) = e(x,y;R)$. Thus we get the assertion.\\
(iii) We first claim that $\ell(R/(x,y)) = e(x,y;R) + \ell(0:_{H^1_{\frak m}(R)}(x,y))$. Indeed, let $R' = R/(x)$. From the short exact sequence
$$0 \to R \overset{x}{\to} R \to R' \to 0$$
we have the exact sequence of local cohomology
$$0 \to H^0_{\frak m}(R') \to H^1_{\frak m}(R) \overset{x}{\to} H^1_{\frak m}(R) \to \cdots.$$
So $H^0_{\frak m}(R') \cong 0:_{H^1_{\frak m}(R)} x$. Since $x$ is a regular
element we have $e(x,y;R) = e(y; R')$. Notice that $\dim R'=1$, we can check that $H^0_{\frak m}(R') \cap (y)R' = yH^0_{\frak m}(R')$. Thus we have the short exact sequence
$$0 \to H^0_{\frak m}(R')/ yH^0_{\frak m}(R') \to R'/(y)R' \to \overline{R}/(y)\overline{R} \to 0,$$
where $\overline{R} = R'/H^0_{\frak m}(R')$. Since $\overline{R}$ is Cohen-Macaulay we have
$$\ell(\overline{R}/(y)\overline{R}) = e(y; \overline{R}) = e(y; R') = e(x,y;R).$$
Therefore
$$\ell(R/(x,y)) = \ell(R'/(y)R') = e(x,y;R) +  \ell(H^0_{\frak m}(R')/ yH^0_{\frak m}(R')).$$
Consider the following exact sequence of finite length modules
$$0 \to 0:_{H^0_{\frak m}(R')}y \to H^0_{\frak m}(R') \overset{y}{\to} H^0_{\frak m}(R') \to H^0_{\frak m}(R')/ yH^0_{\frak m}(R') \to 0.$$
It follows that $\ell(H^0_{\frak m}(R')/ yH^0_{\frak m}(R')) = \ell(0:_{H^0_{\frak m}(R')}y)$. On the other hand, since $H^0_{\frak m}(R') \cong 0:_{H^1_{\frak m}(R)} x$ we have $0:_{H^0_{\frak m}(R')}y \cong 0:_{H^1_{\frak m}(R)}(x,y)$. Thus we have
$$\ell(R/(x,y)) = e(x,y;R) + \ell(0:_{H^1_{\frak m}(R)}(x,y)).$$
Combining the above assertion with (ii) we have
$$\ell((x,y)^{\lim}/(x,y)) = \ell(0:_{H^1_{\frak m}(R)}(x,y)) + \ell (H^1_{\frak m}(R)/(x,y)H^1_{\frak m}(R)).$$
Notice that $0:_{H^1_{\frak m}(R)}(x,y) \cong H^0(x,y;H^1_{\frak m}(R))$ and $H^1_{\frak m}(R)/(x,y)H^1_{\frak m}(R) \cong H^2(x,y;H^1_{\frak m}(R))$. Moreover the Euler charactiristic
$$\chi (x, y; H^1_{\frak m}(R)) = \sum_{i=0}^2 \ell(H^i(x,y;H^1_{\frak m}(R))) = 0$$
since $\dim H^1_{\frak m}(R) = 0 < 2$ (cf. \cite[Theorem 4.7.6]{BH98}). Therefore
$$\ell((x,y)^{\lim}/(x,y)) =\ell(H^1(x,y;H^1_{\frak m}(R))).$$
The proof is complete.
\end{proof}

\section{Some Examples}
The main aim of this section is to compute a certain limit closure. The following example is based on \cite[Example 6.2]{Hu1}.
\begin{example}\label{E5.5}
Let $K$ be a field of characteristic 0 and let $A = K[X,Y,U,V]/(f)$
where $f = XY-UX^2-VY^2$. We denote by $x, y, u, v$ the images of
$X, Y, U, V$, respectively. Set $R = A_{\frak m}$ where $\frak m =
(x, y, u, v)$ and $\frak p = (y,u,v)$. It is easy to prove that $R$
is a Gorenstein three-dimensional domain, and $\frak p$ is a height
two prime ideal of $R$. After completion one can prove that $f$
factors into two formal series $f = (x - vy + \cdots)(y - ux + \cdots)$,
where every term in the element $x-vy+ \cdots$ lies in $(y, u, v)$
except for the first term $x$ (we can prove by induction). These factors is given in more
details in Proposition \ref{P5.9}. There are two minimal
primes lying over $(0)$ in $\widehat{R}$, and $\frak p \widehat{R} +
(x-vy + \cdots) = \widehat{\frak{m}}$. The Lichtenbaum-Hartshore
vanishing theorem implies that $H^3_{\frak p}(R) \neq 0$. By Proposition \ref{P3.3} we have $\mathrm{Ann}\,
H^3_{\frak p}(R) = 0$, so $\cap_{n\geq 1}(y^n,u^n,v^n)^{\lim} = 0$.
However the Hausdorff topology defined by
$\{(y^n,u^n,v^n)^{\lim}\}_{n\geq 1}$ is not equivalent to the $\frak
m$-adic topology. Indeed, by Lemma \ref{L2.5}, $(y^n,u^n,v^n)^{\lim}$
is $\frak m$-primary for all $n$. Hence the $\frak m$-adic topology
is stronger than or equal to the topology defined by $\{(y^n,u^n,v^n)^{\lim}\}_{n\geq 1}$. By Corollary \ref{C5.4} we have $(x-vy+ \cdots) \subseteq
(y^n,u^n,v^n)^{\lim}_{\widehat{R}}$ for all $n\geq 1$. Thus
$(y^n,u^n,v^n)^{\lim}_{\widehat{R}} \nsubseteq \widehat{\frak{m}}^2$
for all $n\geq 1$. Hence $(y^n,u^n,v^n)^{\lim}\nsubseteq
{\frak{m}}^2$ for all $n\geq 1$. This section is devoted to an explicit computation for $(y^n,u^n,v^n)^{\lim}$, $n \ge 1$.
\end{example}

\begin{observation}
Keep all notations as in the previous Example. We have
$(y^n,u^n,v^n)^{\lim}\nsubseteq {\frak{m}}^2$ for all $n\geq 1$. By
the definition of the limit closure for each $n$ there exists $t(n)$
such that
$$(y^{nt(n)+ n},u^{nt(n)+n},v^{nt(n)+ n}):_R
(yuv)^{nt(n)} \nsubseteq {\frak{m}}^2.$$
For all $n = 1, \ldots, 9$, by using computer program we obtain 
$$(y^{2n},u^{2n},v^{2n}):_R(yuv)^n =
(y^n,u^n,v^n,a_n),$$
where $a_n$ is as the following.\\

\begin{tabular}{|c|c|}
  \hline
  $n$ & $a_n$ \\
  \hline
  $1$ & $x$ \\
\hline
  $2$ & $x-yv$ \\
\hline
  $3$ & $x-yv-yuv^2$ \\
\hline
  $4$ & $x-yv-yuv^2 - 2yu^2v^3$ \\
\hline
  $5$ & $x-yv-yuv^2 - 2yu^2v^3 - 5yu^3v^4$ \\
\hline
  $6$ & $x-yv-yuv^2 - 2yu^2v^3 - 5yu^3v^4 - 14u^4v^5$ \\
\hline
  $7$ & $x-yv-yuv^2 - 2yu^2v^3 - 5yu^3v^4 - 14u^4v^5 - 42yu^5v^6$ \\
\hline
  $8$ & $x-yv-yuv^2 - 2yu^2v^3 - 5yu^3v^4 - 14u^4v^5 - 42yu^5v^6 - 132yu^6v^7$ \\
\hline
  $9$ & $x-yv-yuv^2 - 2yu^2v^3 - 5yu^3v^4 - 14u^4v^5 - 42yu^5v^6 - 132yu^6v^7 - 429yu^7v^8$ \\
\hline
\end{tabular}\\
\newline
Following from the above table, we consider the sequence 1, 1, 2, 5, 14, 42,
132, 429. This is the first nine terms, from $C_0$ to $C_8$, of the Catalan sequence.
\end{observation}
\begin{definition}\rm The {\it Catalan numbers}, $C_i$ ($i\geq 0$),
are numbers satisfy the recurrence relation
$$C_0 = 1\,\, \text{and}\,\,\, C_{n+1} = \sum_{i=0}^n C_iC_{n-i}\,\,\,\text{for all}\,\, n \geq 0.$$
\end{definition}
The following is well-known.
\begin{lemma}\label{catalan}
Let $C(t) = \sum_{i=0}^{\infty}C_it^i$ be the generating function
for Catalan numbers. Then
$$C(t) = 1 + tC(t)^2.$$
\end{lemma}

\begin{proposition}\label{P5.9}
With notations as in Example \ref{E5.5}, and $C(t) =
\sum_{i=0}^{\infty}C_it^i$ is the generating function for Catalan numbers. Then
$$(X - YV C(UV)) . (Y - XU C(UV)) = (XY - UX^2 - VY^2) . C(UV).$$
Therefore $(x - yv C(uv)) \cap (y - xu C(uv)) = (0)$ is a reduced
primary decomposition of $(0)$ in $\widehat{R}$.
\end{proposition}
\begin{proof} We have
\begin{eqnarray*}
(X - YV C(UV)) . (Y - XU C(UV)) &=& XY - (UX^2 + VY^2) . C(UV) + XY UV .
C(UV)^2 \\
&=& XY . (1 + UV C(UV)^2) - (UX^2 + VY^2) . C(UV)\\
&=& (XY - UX^2 - VY^2) . C(UV).\quad (\text{By Lemma } \ref{catalan})
\end{eqnarray*}
Hence $(x - yv C(uv)).(y - xu C(uv)) = 0 \in \widehat{R}$. It is
easy to check that $(x - yv C(uv)), (y - xu C(uv))$ are prime ideals
and $(x - yv C(uv)) \cap (y - xu C(uv)) = (0)$.
\end{proof}

The following is the main result of this section.
\begin{proposition}\label{P5.10}
Let $K$ be a field of characteristic 0 and let $A = K[X,Y,U,V]/(f)$
where $f = XY-UX^2-VY^2$. We denote by $x, y, u, v$ the images of
$X, Y, U, V$, respectively. Set $R = A_{\frak m}$ where $\frak m =
(x, y, u, v)$. Then for all $n \geq 1$ we have
$$(y^n,u^n,v^n)^{\lim} = (y^n,u^n,v^n, x- yv \sum_{i=0}^{n-2}C_i(uv)^i),$$
where $C_i$ is the $i$-th Catalan number.
\end{proposition}

\begin{lemma}\label{L5.11}
Keep all notations as in the previous Proposition. Then for all $n
\geq 1$ we have
$$(y^n,u^n,v^n, x- yv \sum_{i=0}^{n-2}C_i(uv)^i) \subseteq (y^{2n},u^{2n},v^{2n}):_R (yuv)^n.$$
\end{lemma}
\begin{proof} It suffices to prove for all $n\geq 1$ that
$$(x- yv \sum_{i=0}^{n-2}C_i(uv)^i)y^n \in (u^n,v^n).$$
By Proposition \ref{P5.9} we have $(x - yv C(uv)).(y - xu C(uv)) =
0$. Hence
$$(x- yv \sum_{i=0}^{n-2}C_i(uv)^i) . (y - xu \sum_{i=0}^{n-2}C_i(uv)^i) \equiv 0 \mod (u^n,v^n).$$
Therefore
$$(x- yv \sum_{i=0}^{n-2}C_i(uv)^i) .  y \equiv (x- yv \sum_{i=0}^{n-2}C_i(uv)^i) . (xu \sum_{i=0}^{n-2}C_i(uv)^i) \mod (u^n,v^n).$$
So
\begin{eqnarray*}
(x- yv \sum_{i=0}^{n-2}C_i(uv)^i) . y^n\,
&\equiv & (x- yv \sum_{i=0}^{n-2}C_i(uv)^i) . (xu \sum_{i=0}^{n-2}C_i(uv)^i)^n\mod (u^n,v^n)\\
&\equiv & 0 \mod (u^n,v^n).
\end{eqnarray*}
The Lemma is proved.
\end{proof}

\begin{proof}[Proof of Proposition \ref{P5.10}] By Lemma
\ref{L5.11}, for all $n\geq 1$ we have
$$(y^n,u^n,v^n, x- yv \sum_{i=0}^{n-2}C_i(uv)^i) \subseteq (y^n,u^n,v^n)^{\lim}.$$
By Proposition \ref{P5.9} we have
$$(X- YV \sum_{i=0}^{n-2}C_i(UV)^i) . (Y - XU \sum_{i=0}^{n-2}C_i(UV)^i) \equiv (XY - UX^2 - VY^2) . C(UV) \mod (U^n,V^n).$$
So $f \in (Y^n,U^n,V^n, X- YV \sum_{i=0}^{n-2}C_i(UV)^i)$. Hence
\begin{eqnarray*}
\ell (R/(y^n,u^n,v^n, x- yv \sum_{i=0}^{n-2}C_i(uv)^i) &=&
\ell (K[X,Y,U,V])/(Y^n,U^n,V^n, X- YV \sum_{i=0}^{n-2}C_i(UV)^i)\\
&=& n^3.
\end{eqnarray*}

\noindent On the other hand
 $$\ell (R/(y^n,u^n,v^n)^{\lim}) = \ell
(\widehat{R}/(y^n,u^n,v^n)^{\lim}_{\widehat{R}}).$$
By
Corollary \ref{C5.4} we have $\cap_{n\geq 1}
(y^n,u^n,v^n)^{\lim}_{\widehat{R}} = (x - yv C(uv))$. Set $R' =
\widehat{R} / (x - yv C(uv))$. By Proposition \ref{P2.7} we have
\begin{eqnarray*}
\ell
(\widehat{R}/(y^n,u^n,v^n)^{\lim}_{\widehat{R}}) &=& \ell
({R'}/(y^n,u^n,v^n)^{\lim}_{R'})\\
&=& \ell (K[[Y,U,V]]/(Y^n, U^n, V^n))
= n^3.
\end{eqnarray*}
 Therefore
$$\ell (R/(y^n,u^n,v^n, x- yv \sum_{i=0}^{n-2}C_i(uv)^i)) = \ell
(R/(y^n,u^n,v^n)^{\lim}).$$
Hnce
$$(y^n,u^n,v^n)^{\lim} = (y^n,u^n,v^n, x- yv \sum_{i=0}^{n-2}C_i(uv)^i)$$
for all $n \ge 1$. The proof is complete.
\end{proof}
The next example shows that the complete condition of $(R,\frak m)$ in Corollary \ref{C5.4} is necessary.
\begin{example}\label{E7.8}\rm
Let $k$ be a field of characteristic 0 and let $A =
K[X,Y,U,V]/((f)\cap (X))$ where $f = XY-UX^2-VY^2$. We denote by $x,
y, u, v$ the images of $X, Y, U, V$, respectively. Set $R = A_{\frak
m}$ where $\frak m = (x, y, u, v)$. By
Proposition \ref{P5.9}
$$(x) \cap (x - yv C(uv)) \cap (y - xu
C(uv)) = (0)$$
is a reduced primary decomposition of $(0)$ in
$\widehat{R}$. Corollary \ref{C5.4} implies that
$$\cap_{n \geq
1}(y^n,u^n,v^n)^{\lim}_{\widehat{R}} = (x) \cap (x - yv C(uv)).$$
So $\cap_{n \geq 1}(y^n,u^n,v^n)^{\lim} = (0)$. But
$$J = \{\frak p \in
\mathrm{Ass}R\mid y,u,v\, \text{ is a system of parameters of}\,
R/\frak p\} = \{(x)\}$$ and $\cap_{\frak p \cap\in J}N(\frak p) =
(x)\neq (0)$.
\end{example}

\end{document}